\documentclass[10pt]{amsart}
\usepackage[leqno]{amsmath}
\usepackage[numbers]{natbib}
\usepackage[mathcal]{euscript}
\usepackage{epsfig}
\usepackage{color}
\usepackage{multicol}
\usepackage{esint}

\usepackage[body={5.5in,8in}, top=1.5in,left=1.50in]{geometry}

\theoremstyle{plain}
\newtheorem{thm}{Theorem}
\newtheorem{cor}{Corollary}

\newtheorem{prop}[cor]{Proposition}
\theoremstyle{definition}

\numberwithin{cor}{section}

\newcommand{\R}{\mathbb R}
\newcommand{\ep}{\varepsilon}

\newcommand{\dist} {\mathrm{dist}}

\newcommand{\Sy}{{\mathcal S_n}}
\newcommand{\iden}{I_n}

\DeclareMathOperator{\tr}{tr}

\begin{document}

\title{Unique continuation for fully nonlinear elliptic equations}
\author{Scott N. Armstrong}
\address{Department of Mathematics\\ The University of Chicago\\ 5734 S. University Avenue
Chicago, Illinois 60637.}
\email{armstrong@math.uchicago.edu}
\author{Luis Silvestre}
\address{Department of Mathematics\\ The University of Chicago\\ 5734 S. University Avenue
Chicago, Illinois 60637.}
\email{luis@math.uchicago.edu}
\date{\today}
\keywords{unique continuation, fully nonlinear elliptic equation}
\subjclass[2010]{35B60}

\begin{abstract}
We show that a viscosity solution of a uniformly elliptic, fully nonlinear equation which vanishes on an open set must be identically zero, provided that the equation is $C^{1,1}$. We do not assume that the nonlinearity is convex or concave, and thus \emph{a priori} $C^2$ estimates are unavailable. Nevertheless, we use the boundary Harnack inequality and a regularity result for solutions with small oscillations to prove that the solution must be smooth at an appropriate point on the boundary of the set on which it is assumed to vanish. This then permits us to conclude with an application of a classical unique continuation result for linear equations. 
\end{abstract}

\maketitle

\section{Introduction}

In this work, we describe a unique continuation result for viscosity solutions of fully nonlinear elliptic equations. We consider the equation
\begin{equation}\label{Feq}
F(D^2u,Du,u) = 0,
\end{equation}
under the assumption that $F$ is uniformly elliptic, $F\in C^{1,1}$ in a neighborhood of the origin, and $F(0,0,0) = 0$. We prove that any viscosity solution $u$ of \eqref{Feq} in a connected domain $\Omega \subseteq \R^n$ which vanishes of infinite order at a point $x_0 \in \Omega$ must be identically zero in $\Omega$. That is, equation \eqref{Feq} possesses the \emph{strong unique continuation} property.

If the solution $u$ was assumed to be $C^3$, we could easily prove our theorem by linearizing the equation and applying a known unique continuation result for linear equations. The difficulty in establishing such a unique continuation result for fully nonlinear equations lies in the fact that the best \emph{a priori} regularity available for viscosity solutions of \eqref{Feq} is $C^{1,\alpha}$, even in the case that $F$ is smooth. We overcome this obstacle with a strategy that combines the boundary Harnack inequality and a regularity result for ``flat" viscosity solutions due to Savin~\cite{S}. These permit us to apply, in neighborhoods of appropriate points in our domain, a classical unique continuation result for linear equations in nondivergence form.

\medskip

We denote by $\Sy$ the set of symmetric $n$-by-$n$ matrices, and consider a nonlinear function
\begin{equation*}
F: \Sy \times \R^n \times \R \to \R.
\end{equation*}
We require $F=F(M,p,z)$ to satisfy the following hypotheses:
\begin{enumerate}
\item[(F1)] $F$ is uniformly elliptic and Lipschitz; that is, there exist constants $0 < \lambda \leq \Lambda$ and $\gamma,\eta \geq 0$, such that for every $M,N\in \Sy$ such that $N\geq 0$, $p,q\in \R^n$ and $z,w\in \R$,
\begin{multline*}
\lambda \| N \| - \gamma |p-q| - \eta |z-w| \leq F(M+N,p,z) - F(M,q,w) \\ \leq \Lambda \| N \| + \gamma |p-q| + \eta |z-w|,
\end{multline*}
\item[(F2)] $F(0,0,0) = 0$, and
\item[(F3)] $F\in C^{1,1}$ in a neighborhood of the point $(0,0,0) \in \Sy\times\R^n\times\R$.
\end{enumerate}

\medskip

A measurable function $u$, defined in a neighborhood of $x_0$, is said to \emph{vanish of infinite order at} $x_0$ provided that for some $\ep > 0$,
\begin{equation} \label{vaninford}
\sup_{\beta > 0}\,  \limsup_{r\to 0}\, r^{-\beta}\int_{B(x_0,r)} |u|^\ep\, dx = 0.
\end{equation}
Clearly $u$ vanishes of infinite order at $x_0$ if it vanishes in a neighborhood of $x_0$.

\medskip

We now state our main result.

\begin{thm}\label{main}
Assume that $F$ satisfies (F1), (F2), and (F3). Suppose that $u$ is a viscosity solution of \eqref{Feq} in a connected open subset $\Omega$ of $\R^n$. Assume that $u$ vanishes of infinite order at some point $x_0 \in \Omega$. Then $u \equiv 0$. 
\end{thm}

To our knowledge, Theorem~\ref{main} is the first unique continuation result for viscosity solutions of general fully nonlinear elliptic equations. Of course, if in addition $F$ is assumed to be concave or convex, then solutions are known \emph{a priori} to be $C^{2,\alpha}$ by the Evans-Krylov theorem. In this case, we may linearize the equation to easily obtain the result, as above.

\medskip

The main idea of our proof of unique continuation is the following. We assume that $\{u=0\}$ has nonempty interior but is not equal to $\Omega$. To obtain a contradiction, it is enough to find one point on the boundary of $\{u=0\}$ around which the solution $u$ is regular enough to linearize the equation in a neighborhood of that point. We prove that this is the case at any point of $\partial \{u=0\}$ satisfying an interior sphere condition. At such a point, we can apply the boundary Harnack inequality to the derivatives of $u$ and then use a result of Savin \cite{S} regarding the regularity of ``flat" solutions of uniformly elliptic equations, obtaining that $u\in C^3$ in a small neighborhood.

\medskip

We remark that Theorem~\ref{main} does not apply to equations of Bellman-Isaacs type, that is, equations of the form
\begin{equation*}
F(D^2u,Du,u) := \inf_{\alpha \in \mathcal I} \sup_{\beta \in \mathcal J} \left( \tr(A_{\alpha\beta} D^2u) + b_{\alpha\beta}\cdot Du + c_{\alpha\beta} u \right) = 0.
\end{equation*}
Such operators are by nature $C^{0,1}$, but not $C^1$. A positively homogeneous operator which is $C^1$ in a neighborhood of the origin is obviously linear, and thus a $C^1$ assumption would be incompatible with nonlinearity. It is not known whether the unique continuation property holds even for the Pucci equations, which are respectively convex and concave, and perhaps the simplest nonlinear equation of Bellman-Isaacs type. 

\medskip

In the next section, we state the preliminary results needed for the proof of Theorem~\ref{main}, which is presented in Section 3.

\section{Preliminaries}

We begin with the statement of a classical strong unique continuation result for the divergence-form linear equation
\begin{equation} \label{lin}
\mathrm{div}(A(x) Du) + b(x) \cdot Du + c(x) u = 0.
\end{equation}
A more general version of the following is proved in H\"ormander \cite{H} (see Theorem 17.2.6) using Carleman estimates. See also Garofalo and Lin \cite{GL}, who obtain the result via monotonicity formulas.

\begin{prop} \label{linearSUC}
Assume that $\Omega$ is a connected domain, the diffusion matrix $A:\Omega \to \Sy$ is uniformly elliptic and Lipschitz, and $b: \Omega \to \R^n$ and $c:\Omega \to \R$ are bounded, measurable functions. Suppose that $u\in W^{1,2}_{\mathrm{loc}}(\Omega)$ is a weak solution of \eqref{lin} in $\Omega$, and that $u$ vanishes of infinite order at some point $x_0 \in \Omega$. Then $u \equiv 0$.
\end{prop}

Notice that Proposition~\ref{linearSUC} may be applied to elliptic equations in nondivergence form if the coefficients are Lipschitz and the solution is sufficiently regular, since in that case we can rewrite the equation in divergence form.

\medskip

Our proof of the strong unique continuation property for solutions of \eqref{Feq} relies in a crucial way on the following regularity result of Savin~\cite{S}, which asserts that any sufficiently small viscosity solution of $\eqref{Feq}$ is a classical solution.

\begin{prop}[Savin \cite{S}] \label{flatreg}
Assume that $F$ satisfies (H1), (H2) and (H3). Then there exists a constant $c_1>0$, depending on $F$, such that if $u$ is a viscosity solution of \eqref{Feq} in $B_1$ satisfying $\sup_{B_1} |u| \leq c_1$, then $u \in C^{2,\alpha}(\overline B_{1/2})$.
\end{prop}

The estimate in \cite[Theorem 1.3]{S} actually requires that $F\in C^2$ in a neighborhood of the origin, but by inspecting the proof we see that it depends only on $n$, the constants in (F1), and the maximum of $|D^2F|$ near the origin. We therefore obtain the result for $F$ only $C^{1,1}$ near the origin via a standard regularization procedure. 

\medskip

As previously mentioned, the best regularity available for solutions of general uniformly elliptic equation is $C^{1,\alpha}$. A proof of the following interior $C^{1,\alpha}$ regularity assertion in the case $F=F(M)$ can be found in \cite{CC}, see Corollary 5.7. For $F$ with dependence on lower-order terms, we refer for example to Trudinger \cite{T}.

\begin{prop} \label{c1alph}
Assume that $F$ satisfies (F1). Suppose that $u$ is a bounded viscosity solution of \eqref{Feq} in $B_1$. Then $u \in C^{1,\alpha}(\overline B_{1/2})$ for a constant $0 < \alpha <1$ depending only on $n$, $\lambda$, $\Lambda$, $\gamma$, and $\eta$.
\end{prop}

A key ingredient in the proof of Theorem~\ref{main} is a H\"older regularity result up to the boundary for solutions of uniformly elliptic equations. Before we state this, let us define the extremal operators
\begin{equation*}
G^\pm (M,p,z): = \mathcal M^\pm (M) \pm \gamma |p| \pm \eta |z|,
\end{equation*}
where the operators $\mathcal M^\pm$ are the usual Pucci operators, defined by
\begin{equation*} 
\mathcal M^+(M) : =  \sup_{\lambda \iden \leq A \leq \Lambda\iden} \tr (AM) \quad \mbox{and} \quad \mathcal M^-(M) : =  \inf_{\lambda \iden \leq A \leq \Lambda\iden} \tr (AM).
\end{equation*}
The condition (F1) can be written equivalently as
\begin{equation} \label{Fsublin}
G^-(M-N,p-q,z-w) \leq F(M,p,z) - F(N,q,w) \leq G^+(M-N,p-q,z-w),
\end{equation}
for all $M,N\in \Sy$, $p,q\in \R^n$, and $z,w\in \R$. We note that the operators $G^\pm$ satisfy the hypotheses (F1), (F2), and (F3). In light of (F2), we see that $G^-(M) \leq F(M) \leq G^+(M)$, and therefore any solution of $F=0$ is both a supersolution of $G^- =0$ and a subsolution of $G^+=0$.

\begin{prop} \label{bh}
Let $\Omega$ be a smooth domain, and $\Gamma$ a relatively open subset of the boundary $\partial \Omega$. Suppose that $u \in C(\Omega\cup \Gamma)$ is a viscosity solution of the differential inequalities
\begin{equation} \label{SLL}
G^-(D^2u,Du,u) \leq 0 \leq G^+(D^2u,Du,u) \quad \mbox{in} \ \Omega,
\end{equation}
and $u = 0$ on $\Gamma$. Then there exists a $C^\alpha$ function $K:\Gamma \to \R^n$ such that at each point $x_0 \in \Gamma$, we have
\begin{equation*}
\left| u(x) - K(x_0) \cdot (x-x_0) \right| \leq C_{x_0} |x-x_0|^{1+\alpha}, \quad x \in \Omega.
\end{equation*}
The constant $\alpha> 0$ depends only on $n$, $\lambda$, $\Lambda$, $\gamma$, and $\eta$, and the constant $C_{x_0}$ depends additionally on $\dist(x_0, \partial \Omega \setminus \Gamma)$. The function $K$ is then the normal derivative to $u$ at the boundary $\Gamma$.
\end{prop}

Proposition~\ref{bh} follows from the boundary Harnack inequality, first observed for nondivergence form equations by Bauman \cite{B} and Krylov \cite{K}. It is difficult to find a good reference for the latter result in the generality we require, although we remark that it can be obtained from very straightforward modifications to the proof of \cite[Theorem 1.3]{BCN} in the linear setting, relying on the extremal operators and the natural sublinearity \eqref{Fsublin} in place of linearity. This is explained in some detail in the appendix of the recent preprint~\cite{ASS}.

\section{The strong unique continuation property}

There are two main steps in the proof of Theorem~\ref{main}. We first prove the unique continuation property, that is, we obtain the conclusion of the theorem under the additional assumption that the solution $u$ of \eqref{Feq} vanishes identically in an open subset of $\Omega$. The second step is to reduce the strong unique continuation property to the weaker unique continuation property, which consists of showing that a solution vanishing of infinite order at a point must be identically zero in a neighborhood of that point. Both arguments depend on Propositions~ \ref{linearSUC}, \ref{flatreg}, and \ref{c1alph}, while in the first step we also rely crucially on Proposition~\ref{bh}, and the second step requires the local maximum principle.

\begin{proof}[Proof of Theorem~\ref{main}]

Consider a solution $u \in C^{1,\alpha}_{\mathrm{loc}}(\Omega)$ of \eqref{Feq} in a domain $\Omega$.

\medskip

\noindent\emph{Step 1. Proof of the unique continuation property.}

\medskip

Assume that $u$ vanishes identically on an open subset of $\Omega$, an assumption we remove in Step 2, below. Let $W$ be a connected component of the interior of $\{ x \in \Omega : u(x) = 0\}$, and suppose on the contrary that $W \neq \Omega$. Then we can find a ball $B \subseteq W$ and a point $x_0\in \partial B \cap \partial W$. We derive a contradiction by showing that $u$ vanishes in a neighborhood of the point $x_0$.

Select a unit direction $e\in \R^n$, $|e|=1$, and observe that the function $u_e : = e\cdot Du$ satisfies the differential inequalities \eqref{SLL}, which follows from the fact that $u(x+he) - u(x)$ satisfies \eqref{SLL} for each $h> 0$, as remarked for example in \cite[Proposition 5.5]{CC}. Since $u\in C^{1,\alpha}$ and $u$ vanishes on $B$, it is clear that $u_e$ vanishes on $\overline B$. Applying Proposition~\ref{bh}, we discover that for some $k\in \R^n$,
\begin{equation}
|u_e(x) - k \cdot (x-x_0)| \leq C |x-x_0|^{1+\alpha}, \quad x\in \Omega \setminus B.
\end{equation}
Since $u_e$ vanishes on $B$, we deduce that $k=0$. Indeed, it suffices to consider a slightly smaller ball $\tilde B$, concentric to $B$, at a point $\tilde x_0$ near $x_0$, and argue by continuity. 

We have shown that $|u_e(x)| \leq C|x-x_0|^{1+\alpha}$ in a neighborhood of $x_0$ for every direction $|e|=1$, and so we conclude that
\begin{equation*}
|u(x)| \leq C|x-x_0|^{2+\alpha}, \quad x\in\Omega.
\end{equation*}
By rescaling the solution and applying Proposition~\ref{flatreg} in a small neighborhood of $x_0$, we deduce that $u\in C^{2,\alpha}(\overline B(x_0,r))$ for some small $r > 0$. Clearly $D^2u(x_0) = 0$. Using (F3), we see that by shrinking $r$, if necessary, we may assume that for each point $x\in B(x_0,r)$, the triple $(D^2u(x),Du(x),u(x))$ lies in the neighborhood of $(0,0,0)$ on which $F$ is $C^{1,1}$. It then follows from Schauder estimates that $u\in C^{3,\alpha}(B(x_0,r))$, and we may differentiate \eqref{Feq} to obtain that
\begin{equation} \label{diffFeq}
\frac{\partial F}{M_{ij}} \frac{\partial^2 u_e}{\partial x_i \partial x_j} + \frac{\partial F}{\partial p_i}\frac{\partial u_e}{\partial x_i} + \frac{\partial F}{\partial z}u_e= 0 \quad \mbox{in} \ B(x_0,r).
\end{equation} 
We may also write \eqref{diffFeq} in divergence form as
\begin{equation} \label{df-diffFeq}
\mathrm{div}\!\left( A(x) Du \right) + b(x) \cdot Du + c(x) u = 0\quad \mbox{in} \ B(x_0,r), 
\end{equation}
where the coefficients $A(x) = (a_{ij}(x))$ and $b(x) = (b_1(x),\ldots, b_n(x))$ and $c(x)$ are given by
\begin{equation*}
\left\{ 
\begin{aligned}
& a_{ij}(x) : = \frac{\partial F}{M_{ij}}(D^2u(x),Du(x),u(x)), \\
& b_i(x) : = - \frac{\partial a_{ij}}{\partial x_j}(D^2u(x),Du(x),u(x))(x) + \frac{\partial F}{\partial p_i}(D^2u(x),Du(x),u(x)), \\
& c(x) : = \frac{\partial F}{\partial z}(D^2u(x),Du(x),u(x)).
\end{aligned} \right.
\end{equation*}
It is clear that $A$ is Lipschitz, while $b$ and $c$ are bounded. Therefore we may apply Proposition~\ref{linearSUC} to conclude that $u_e\equiv 0$ in $B(x_0,r)$. Since the latter holds for every $|e|=1$ and $u(0)=0$, we conclude that $u\equiv 0$ in $B(x_0,r)$. The claim is proved, the desired contradiction having been derived.

\medskip

\noindent\emph{Step 2. Proof of the strong continuation property.}

\medskip

To remove the assumption that $u$ vanishes on an open set, we suppose instead that 
$x_0 \in \Omega$ is a point at which $u$ vanishes of infinite order, and proceed to show that $u$ vanishes in a neighborhood of $x_0$.  From \eqref{vaninford} we have $u(x_0) = 0$ and $Du(x_0) = 0$.

The local maximum principle (see \cite[Theorem 4.8]{CC} for $F=F(M)$, or consult \cite{BS} for lower-order terms) and \eqref{vaninford} imply that for all $\beta > 0$,
\begin{equation*}
\sup_{B_r} |u| \leq C \left(\fint_{B_{2r}} |u|^\ep \, dx\right)^{1/\ep} = O\!\left( r^\beta\right).
\end{equation*}
Since the latter holds in particular for some $\beta > 2$, we apply Proposition~\ref{flatreg} to deduce that $u \in C^{2,\alpha}(B(x_0,r))$ for some sufficiently small $r > 0$. Using again that $u$ vanishes of infinite order at $x_0$, we see that $D^2u(x_0) = 0$. We now argue as in Step 1 above to obtain $u\equiv 0$ in $B(x_0,r)$, after possibly shrinking~$r$.
\end{proof}

\subsection*{Acknowledgment}
The first author was partially supported by NSF Grant DMS-1004645, and the second author was partially supported by NSF grant DMS-1001629 and the Sloan Foundation.

\bibliographystyle{plain}
\bibliography{uniqcont}

\end{document}